\documentclass[12pt,a4paper]{article}

\usepackage{amssymb, amsmath, amsthm}
\usepackage{url}
\usepackage[utf8]{inputenc}

\def\SetMargins#1#2#3#4{

\global\topmargin -1.5 true in
\global\advance\topmargin #1

\global \textheight \paperheight
\global \advance\textheight -#1
\global \advance\textheight -#2

\global \textwidth \paperwidth
\global \advance\textwidth -#3
\global \advance\textwidth -#4

\global \oddsidemargin -1 true in
\global \advance\oddsidemargin #3
\global \evensidemargin -1 true in
\global \advance\evensidemargin #4
}

\SetMargins{25mm}{25mm}{25mm}{25mm}

\newtheorem{Thm}{Theorem}
\newtheorem*{Thm*}{Theorem}

\newtheorem*{Lem*}{Lemma}

\newtheorem*{Cor*}{Corollary}

\newtheorem*{Rem*}{Remark}

\def\I#1{\left<#1\right>}

\begin{document}


\begin{center}
  \textbf{\large Representing the GCD as linear combination in non-PID
    rings}%
  \footnote{Key words and phrases: unique factorization domain,
    greatest common divisor, pricipal ideal }${}^,$%
  \footnote{2000 Mathematics Subject Classification:
    41A44 (approximations and expansions, best constants)}\\
  \medskip Géza Kós%
  \footnote{Computer and Automation Research Institute, Budapest and
    Loránd Eötvös University, Budapest.  E-mail:
    \texttt{kosgeza@sztaki.hu}}%

  \vspace{5mm}

  (Submitted to Acta Mathematica Hungarica on June 29, 2012)

  \vspace{5mm}

\end{center}


\begin{abstract}
  In this note we prove the following fact: if finite many elements
  $p_1,p_2,\ldots,p_n$ of a unique factorization domain are given such that
  the greatest common divisor of each pair $(p_i,p_j)$ can be expressed as a
  linear combination of $p_i$ and $p_j$ then the greatest common divisor of
  all $p_i$s also can be expressed as a linear combination of
  $p_1,\ldots,p_n$. We prove am analogous statement in commutative rings.
\end{abstract}


In our previous work \cite{ref:periodic} we needed the following result.
\begin{Thm}[Corollary 1.4. in \cite{ref:periodic}]
  \label{thm0}
  Suppose that the positive integers\break $d_1,d_2,\dots,d_n$ divide
  the positive integer $D$, and let $p_i(x)=
  1+x^{d_i}+x^{2d_i}+\ldots+x^{D-d_i}=\frac{1-x^D}{1-x^{d_i}}$ for
  $1\le i\le n$.  Then there exist polynomials $w_i\in\mathbb{Z}[x]$
  such that the greatest common divisor of the polynomials $p_i$ can
  be expressed as
  $\gcd(p_1,p_2,\ldots,p_n)=w_1p_1+w_2p_2+\ldots+w_np_n$.
\end{Thm}

If the ring $\mathbb{Z}[x]$ is replaced by $\mathbb{Q}[x]$, this
statement becomes trivial, because $\mathbb{Q}[x]$ is a principal
ideal domain (which is not the case for $\mathbb{Z}[x]$).  The
statement is obvious for $n=2$ also, because
$p_i=\frac{1-X^D}{(1-x^{a_1})(1-x^{a_2})}(1-x^{a_{i+1}})$ and Euclid's
algorithm expresses $\gcd(1-x^{a_1},1-x^{a_2})$ as a linear
combination of $1-x^{a_1}$ and $1-x^{a_2}$, with coefficients from
$\mathbb{Z}[x]$.

In this paper we present two theorems which can be replacements for
Theorem~\ref{thm0}. Theorem~\ref{thm1} is valid in unique factorziation
domains, and uses only the property that for every pair $i,j$ of indexes,
$\gcd(p_i,p_j)$ can be expressed as a linear combination of $p_i$ and $p_j$.

Theorem~\ref{thm2} works in commutative rings (without unique
factorzation), and expresses the greatest common divisor of the
$(n-1)$-factor products\break $p_1\cdots p_{i-1}p_{i+1}\cdots p_n$
($i=1,2,\ldots,n$).

The results are formulated in equivalent forms, for greatest common divisors
of elements and for principal ideals in a ring.


\begin{Thm}\label{thm1}
  Let $R$ be a unique factorization domain.
  
  \begin{itemize}
  \item[(a)] Let $p_1,p_2,\ldots,p_n$ be nonzero elements
    in~$R$. Suppose that the greatest common divisor of every pair of
    $p_1,p_2,\ldots,p_n$ can be represented as a linear combination of
    them, i.e. for every $1\le i<j\le n$, we have
    $\gcd(p_i,p_j)=u_{ij}p_i+u_{ji}p_j$ with some $u_{ij},u_{j,i}\in
    R$. Then there exist elements $w_1,w_2,\ldots,w_n$ in $R$ such
    that $\gcd(p_1,p_2,\ldots,p_n)=w_1p_1+w_2p_2+\ldots+w_np_n$.
    
  \item[(b)] Let $I_1,I_2,\dots,I_n$ be principal ideals in $R$ such
    that $I_i+I_j$ is a principal ideal for every pair of indices
    $1\le i,j\le n$. Then $I_1+I_2+\dots+I_n$ also is a principal
    ideal.
  \end{itemize}
\end{Thm}


\begin{Thm}\label{thm2}
  Let $R$ be a commutative ring with unity and $n\ge 2$.

  \begin{itemize}
  \item[(a)] Let $p_1,p_2,\ldots,p_n\in R$. If every nonempty subset
    of $\{p_1,p_2,\ldots,p_n\}$ generates a principal ideal then the
    $n$ products $p_2p_3p_4\ldots p_n$, 
    \break
    $p_1p_3p_4\ldots p_n$, \ldots,
    $p_1p_2p_3\ldots p_{n-1}$ gereate a principal ideal.

  \item[(b)] If $I_1,I_2,\dots,I_n$ are principal ideals in $R$ such
    that for $\sum\limits_{h\in H}I_h$ is a principal ideal for every
    nonempty set $H\subset\{1,2,\dots,n\}$, then
    $$
    I_2I_3I_4\ldots I_n + I_1I_3I_4\ldots I_n +\dots+ I_1I_2\ldots I_{n-1}
    $$
    also is a principal ideal.
  \end{itemize}
\end{Thm}


\begin{proof}[Proof of Theorem~\ref{thm1}]
  We prove by induction on~$n$. The cases $n=1$ and $n=2$ are trivial.
  
  For the case $n=3$ we follow the notation in
  statement~(a). Let\break ${\gcd(p_1,p_2,p_3)=d}$,
  $\gcd(p_2,p_3)=de_1$, $\gcd(p_1,p_3)=de_2$ and
  $\gcd(p_1,w_2)=de_3$. The elements $e_1,e_2,e_3$ are pairwise
  co-prime, because
  \begin{eqnarray}
    \gcd(de_i,de_{i+1})
    &=&\gcd(\gcd(p_{i+1},p_{i+2}),\gcd(p_i,p_{i+2}))= \notag  \\
    &=& \gcd(p_1,p_2,p_3)=d. \quad (i=1,2,3)
    \label{geza:gcdgcd}
  \end{eqnarray}
  By the construction, $de_{i+1}$ and $de_{i+2}$
  divide~$p_i$. Then~(\ref{geza:gcdgcd}) implies that
  $de_{i+1}e_{i+2}$ also divides~$p_i$, so $p_i=de_{i+1}e_{i+2}f_i$
  with some $f_i\in R$. Hence,
  \begin{eqnarray*}
    de_{i+2} &=& \gcd(p_i,p_{i+1}) = u_{i,i+1}p_i+u_{i+1,i}p_{i+1} = \\
    &=& u_{i,i+1}de_{i+1}e_{i+2}f_i+u_{i+1,i}de_ie_{i+2}f_{i+1}
  \end{eqnarray*}
  and therefore
  $$u_{i,i+1}e_{i+1}f_i+u_{i+1,i}e_if_{i+1}=1. \quad (i=1,2,3) $$
  Applying this, we can find that
  \begin{align*}
    d &= d(u_{12}e_2f_1+u_{21}e_1f_2) \\
    &= d\Big(u_{12}e_2f_1(u_{13}e_3f_1+u_{31}e_1f_3) +
    u_{21}e_1f_2(u_{23}e_3f_2+u_{32}e_2f_3)\Big) \\
    &= (u_{12}u_{13}f_1) (de_2e_3f_1) + (u_{21}u_{23}f_2) (de_1e_3f_2)
    + (u_{12}u_{31}f_1 + u_{21}u_{32}f_2) (de_1e_2f_3) \\
    &= (u_{12}u_{13}f_1) p_1 + (u_{21}u_{23}f_2) p_2
    + (u_{12}u_{31}f_1 + u_{21}u_{32}f_2) p_3.
  \end{align*}
  So the elements $w_1=u_{12}u_{13}f_1$, $w_2=u_{21}u_{23}f_2$,
  $w_3=u_{12}u_{31}f_1 + u_{21}u_{32}f_2$ satisfy
  $d=w_1p_1+w_2p_2+w_3p_3$.

  \medskip

  For the induction step we use the more convenient notatiopn from part~(b).
  Let $n\ge4$ and assume that the theorem is valid for all smaller values. Let
  $J=I_{n-1}+I_n$ and apply the theorem for the ideals $I_1,\dots,I_{n-2}$ and
  $J$.  From the case $n=3$ we know that $I_i+J=I_i+I_{n-1}+I_n$ is a
  principal ideal and hence the induction hypothesis can be applied.
\end{proof}


\bigskip

For the proof of Theorem~\ref{thm2} we will use the following tool.

\begin{Lem*}
  \begin{itemize}
  \item[(a)] If $a,b\in R$ and $\I{a,b}=\I{d}$ for some $d\in R$, then
    there exists an element $m\in R$ such that $\I{a}\cap\I{b}=\I{m}$
    and $md=ab$.
  \item[(b)] If $A$, $B$ and $A+B$ are principal ideals in $R$ then
    $A\cap B$ also is a principal ideal and ${(A+B) \cdot (A\cap
      B) = AB}$.
  \end{itemize}
\end{Lem*}

\begin{proof}
  Let $A=\I{a}$, $B=\I{b}$ and $A+B=\I{d}$. Then there are some
  elements $u,v,p,q\in R$ for which $d=au+bv$, $a=pd$ and $b=qd$. We
  show that $A\cap B=\I{pqd}$.

  Since $\I{pqd}\subset\I{pd}=\I{a}=A$ and
  $\I{pqd}\subset\I{qd}=\I{b}=B$, we have $\I{pqd}\subset A\cap B$.
    
  For the converse relation take an arbitrary $w\in A\cap B$; we have
  to prove $w\in\I{pqd}$. Chose two elements $s,t\in R$ with the
  property $w=as=bt$. Then
  \begin{eqnarray*}
    pds &=& as = w = bt = qdt, \quad\text{and}\\
    w &=& pds = p(au+bv)s =
    pu\cdot as + pvs\cdot b = \\
    &=& pu\cdot qdt + pvs\cdot qd =
    pqd\cdot(tu+vs) \in\I{pqd}.
  \end{eqnarray*}
  Therefore $A\cap B\subset\I{pqd}$ also holds, and $A\cap B=\I{pqd}$
  indeed.

  Finally,
  $$
  (A+B) \cdot (A\cap B) = \I{d} \cdot \I{pqd} = 
  \I{pd} \cdot \I{qd} = A \cdot B.
  $$
\end{proof}


\begin{Rem*}
  The Lemma is an extension of the basic property of principal ideal domains
  that the product of the greatest common divisor and the least common
  multiple of two elements is associated with the product of the two elements.

  The converse statement is false. For example, in the polynomial ring
  $\mathbb{R}[x,y]$ we have $\I{x}\cap\I{y}=\I{xy}$, but the ideal
  $\I{x,y}$ is not a principal ideal.
\end{Rem*}


\begin{proof}[Proof of Theorem~\ref{thm2}]
  For arbitrary ideals $J_1,J_2,\ldots,J_k$, define
  $$
  S(J_1,\ldots,J_k) = J_2J_3J_4\ldots J_k + 
  J_1J_3J_4\ldots J_{k-1} + \ldots + J_1J_2\ldots J_{k-1}.
  $$
  We show by induction on $n$ that $S(I_1,I_2,\ldots,I_n)$ is a
  principal ideal. For $n=2$ the statement is trivial.  Assume that
  $n\ge3$ and the theorem is true for smaller values. Let
  $$
  F=I_1+I_2 \quad\text{and}\quad
  M=S(I_1,I_3,I_4,\ldots,I_n)\cap S(I_2,I_3,I_4,\ldots,I_n).
  $$
  By the assumptions of the statement, $F$ is a principal ideal. We
  show that

  \begin{itemize}
  \item[(1)] $M$ is a principal ideal;
  \item[(2)] $ S(I_1,I_2,\ldots,I_n) = FM$.
  \end{itemize}
  
  By the induction hypothesis $S(I_1,I_3,I_4,\ldots,I_n)$ and
  $S(I_2,I_3,I_4,\ldots,I_n)$ are principal ideals; in view of the
  Lemma, to prove (1) it is sufficient to show that their sum also is
  principal.
  \begin{gather*}
    S(I_1,I_3,I_4,\ldots,I_n) + S(I_2,I_3,I_4,\ldots,I_n) = \\
    = \big( I_1\cdot S(I_3,\ldots,I_n)+I_3I_4\ldots I_n\big)
    + \big( I_2\cdot S(I_3,\ldots,I_n)+I_3I_4\ldots I_n\big) = \\
    = (I_1+I_2)\cdot S(I_3,\ldots,I_n)+I_3I_4\ldots I_n = \\
    = F\cdot S(I_3,\ldots,I_n)+I_3I_4\ldots I_n
    = S(F,I_3,I_4,\ldots,I_n).
  \end{gather*}
  Applying the induction hypothesis to the ideals
  $F=I_1+I_2,I_3,\ldots,I_n$ we get that $S(F,I_3,I_4,\ldots,I_n)$
  is principal. Hence (1) holds.
  
  \medskip

  The product $I_3I_4\ldots I_n$ is listed in the sums defining 
  $S(I_1,I_3,I_4,\ldots,I_n)$ and $S(I_2,I_3,I_4,\ldots,I_n)$. Hence,
  \begin{gather*}
    I_1I_3I_4\ldots I_n \subset
    I_1 \cdot \Big(S(I_1,I_3,I_4,\ldots,I_n)\cap
    S(I_2,I_3,I_4,\ldots,I_n)\Big)
    \subset \\
    \subset
    (I_1+I_2) \cdot
    \Big(S(I_1,I_3,I_4,\ldots,I_n)\cap S(I_2,I_3,I_4,\ldots,I_n)\Big)
    = FM.
  \end{gather*}
  and analogously $I_2I_3I_4\ldots I_n \subset FM$.
  
  By the Lemma we have $(I_1+I_2)\cdot(I_1\cap
  I_2)=I_1I_2$. Therefore, for arbitrary $3\le
  i_1<i_2<\ldots<i_{n-3}\le n$ we have
  \begin{gather*}
    I_1I_2I_{i_1}I_{i_2}\ldots I_{i_{n-3}}\subset
    I_1I_2\cdot S(I_3,I_4,\ldots,I_n) = \\
    = (I_1+I_2) \cdot
    \big(( I_1\cap I_2)\cdot S(I_3,I_4,\ldots,I_n)  \big) \subset \\
    \subset (I_1+I_2) \cdot \Big(\big( I_1\cdot S(I_3,I_4,\ldots,I_n)\big)\cap
    \big( I_2\cdot S(I_3,I_4,\ldots,I_n)\big)\Big) \subset \\
    \subset (I_1+I_2) \cdot \big(S(I_1,I_3,I_4,\ldots,I_n)\cap
    S(I_2,I_3,I_4,\ldots,I_n)\big) = FM.
  \end{gather*}
  Now we have proved that the ideals $I_2I_3I_4\ldots I_n$,
  $I_1I_3I_4\ldots I_n$, \dots, $I_1I_2\ldots I_{n-1}$ are all subsets
  of $FM$, so
  $$
  S(I_1,I_2,\ldots,I_n) = I_2I_3I_4\ldots I_n + I_1I_3I_4\ldots I_n
  +\dots+ I_1I_2\ldots I_{n-1} \subset FM.
  $$

  The converse relation also holds, since
  $$ 
  FM =
  \big(S(I_1,I_3,I_4,\ldots,I_n)\cap
  S(I_2,I_3,I_4,\ldots,I_n)\big)\cdot( I_1+ I_2)
  \subset
  $$
  $$
  \subset
  S(I_2,I_3,I_4,\ldots,I_n)\cdot I_1+
  S(I_1,I_3,I_4,\ldots,I_n)\cdot I_2
  \subset S(I_1,I_2,\ldots,I_n).
  $$
  Therefore, $S(I_1,I_2,\ldots,I_n)=FM$, the statement (2) holds. Since
  \break
  $S(I_1,I_2,\ldots,I_n)=FM$ is the product of two principal ideals, 
  $S(I_1,I_2,\ldots,I_n)$ also is a principal ideal.
\end{proof}



\end{document}